\documentclass{article}
\usepackage{graphicx}
\usepackage{authblk}
\usepackage{color}
\usepackage{amsmath, amsthm, amssymb}
\usepackage{bm}
\usepackage{enumitem}
\usepackage{bbm}
\usepackage{blindtext}
\usepackage{titlesec, blindtext, color}
\usepackage{amsthm}
\usepackage[a4paper, total={5.8in, 8in}]{geometry}
\usepackage{mathtools}
\usepackage{tikz}
\usepackage{caption}
\usepackage{subcaption}
\usepackage{tikz-cd}
\usepackage{varwidth}
\usepackage{tasks}
 \usepackage [latin1]{inputenc}
\usepackage{hyperref}
\hypersetup{
colorlinks=true,
citecolor=black,
linkcolor=black,
linktoc=all
}

\newtheorem{theorem}{Theorem}[section]

\newtheorem{lemma}[theorem]{Lemma}

\newtheorem{conjecture}[theorem]{Conjecture}

\theoremstyle{definition}
\newtheorem{definition}[theorem]{Definition}
\newtheorem{remark}[theorem]{Remark}

\newcommand{\SortNoop}[1]{}

\begin{document}

\title{Coexistence in competing first passage percolation with conversion}
\author{Thomas Finn\footnote{tjfinn@bath.ac.uk, University of Bath, Deptartment of Mathematical Sciences, supported by a scholarship from the EPSRC Centre for Doctoral Training in Statistical Applied Mathematics at Bath (SAMBa), under the project EP/L015684/1.} \quad Alexandre Stauffer\footnote{astauffer@mat.uniroma3.it, Universit\`a Roma Tre, Dipartimento di Matematica e Fisica; University of Bath, Department of Mathematical Sciences, supported by EPSRC Fellowship EP/N004566/1.}}
\date{}

\maketitle

\begin{abstract}
	We introduce a two-type first passage percolation competition model on infinite connected graphs as follows.
	Type 1 spreads through the edges of the graph at rate $1$ from a single distinguished site, while all other sites are initially vacant. 
	Once a site is occupied by type 1, it converts to type 2 at rate $\rho>0$.
	Sites occupied by type 2 then spread at rate $\lambda>0$ through vacant sites \emph{and} sites occupied by type 1, whereas type 1 can only spread through vacant sites.
	If the set of sites occupied by type 1 is non-empty at all times, we say type 1 \emph{survives}.
	In the case of a regular $d$-ary tree for $d\geq 3$, we show type 1 can survive when it is slower than type 2, provided $\rho$ is small enough.  
	This is in contrast to when the underlying graph is $\mathbb{Z}^d$, where for any $\rho>0$, type 1 dies out almost surely if $\lambda>1$.
\end{abstract}

\section{Introduction}

	Consider the following two-type first passage percolation model on an infinite connected graph $G$.
	Each site can either be occupied by type 1, type 2 or be vacant according to the following dynamics. 
	At time $0$,  a distinguished site is occupied by type 1 while every other site is vacant. 
	Sites occupied by type 1 attempt to occupy neighbouring vacant sites at rate 1. 
	Once a site is occupied by type 1, it is converted to type 2 at rate $\rho>0$. 
	That is, we define the collection of random variables $\left\{\mathcal{I}_x\right\}_{x\in G}$ of conversion times, that are i.i.d.\ exponential random variables of rate $\rho>0$, assigned to each site of $G$.
	Once a site is occupied by type 1, it waits for its respective conversion time to expire before converting to type 2. 
	Sites occupied by type 2 then spread type 2 to vacant sites \emph{and} sites occupied by type 1 at rate $\lambda>0$.
	
	This model can be seen as a variant of the chase-escape dynamics in predator-prey models (see Section~\ref{sec:related_work} for more details).
	In these models, there is initially a single predator that evolves to block the spread of a species of prey.
	A natural interpretation for our model is as a spreading infection where individuals are either aware or unaware of their infected status. 
	Unaware individuals spread the infection to nearby individuals but become aware of their infected status after a certain time elapses.
	Aware individuals try to warn neighbouring individuals which become aware of the spread of infection even if it it has not been infected. 
	Uninfected individuals that are aware take the necessary measures (for example, self-isolating) to avoid infection and do not get infected.
	Can the unaware individuals coexist with aware individuals for all time?
	Equivalently, can the infection reach an unbounded number of individuals?
	
	Another motivation for us to introduce this model comes from a recent way of analysing strongly interacting particle systems through growth models; for example, the analysis of multiparticle diffusion limited 
	aggregation by Sidoravicius and Stauffer~\cite{sidoravicius2019multi}
	and the analysis of a heterogeneous spread of infection model by Dauvergne and Sly~\cite{dauvergne2021spread}. 
	We believe the competition process we introduce is a natural model for such applications; we discuss this further in Section~\ref{sec:related_work}. 
	
	The main interest of this paper is understanding coexistence regimes in this model on different graphs.
	We prove on the regular tree, type 1 can survive even if it is slower than type $2$ (i.e.\ $\lambda$ is larger than one), so long as $\rho$ is sufficiently small (c.f.\ Theorem~\ref{thm:crit_lambda_tree}).
	Then, we prove such behaviour on the lattice is impossible and type 1 dies out even if it is just faster than type 2, for all $\rho>0$ (c.f.\ Theorem~\ref{thm:lattice}).
	
	A major difficulty in analysing this model is the counter-intuitive lack of monotonicity.
	If we increase $\rho$ or $\lambda$, it seems we can only decrease the probability that type 1 survives.
	Surprisingly, proving this remains an open problem.  
	The issue is the model is non-monotone in the sense that the standard coupling argument fails to hold (unlike other competition models like the two-type Richardson model as discussed in Section~\ref{sec:related_work}).
	Models that lack monotonicity require a careful analysis as they include the possibility of many phase transitions occurring. 
	The subtle behaviour of processes that lack monotonicity has been studied in related models in Candellero and Stauffer \cite{candellero2020} and Deijfen and H\"aggstr\"om \cite{def}.
	Another difficulty the model poses is non-equilibrium dynamics and long-range correlations between the occupancy of sites.
	For example, to determine whether a site is ever occupied by type 1, one may need non-local information about the first times other sites are occupied by type 1 and how type 2 spreads from them.
	
\subsection{Our results}	
\label{sec:results}
	The main result of this paper is the following.
	On the $d$-ary tree for $d\geq3$ (i.e.\ the infinite tree where all vertices have \emph{degree} equal to $d$), we show type 1 can survive even in a regime where it is slower than type 2.
	
\begin{theorem}
\label{thm:crit_lambda_tree}
	Fix $d\geq3$ and consider the $d$-ary tree.
	There exists $\lambda_0=\lambda_0(d)>1$ such that if $\lambda\in(0,\lambda_0)$ and $\rho$ is small enough, then type 1 survives with positive probability.
\end{theorem}

	Recalling the interpretation of the model as the spread of infection and awareness, we deduce the following from Theorem~\ref{thm:crit_lambda_tree}.
	In the case of the $d$-ary tree, the infection can survive even if awareness spreads faster than the infection provided infected individuals do not become aware until a typically large time has passed.

\begin{remark}
	Theorem~\ref{thm:crit_lambda_tree} may naturally be extended to all supercritical Galton-Watson trees of uniformly bounded degree. 
	The requirement of uniformly bounded degrees is needed to appeal to a result in branching random walks by Addario-Berry and Reed \cite{addario2009minima} (see Theorem~\ref{thm:BRW}). 
	The behaviour of the model on more general trees remains an open problem.
	For example, if the degree distribution follows a power law, can type 1 survive for all $\lambda\in(0,\infty)$?
\end{remark}
	
	Our next result is establishing the behaviour on the regular tree is fundamentally different than on the lattice.
	More precisely, on $\mathbb{Z}^d$ for $d\geq2$,  for all $\rho>0$, there exists a constant \emph{smaller} than one such that if $\lambda$ is larger than this constant, type 1 dies out almost surely.
	In other words, there exists a regime where type 1 is faster than type 2 but still dies out almost surely.
	We note it is immediate for $d=1$ that type 1 dies out almost surely for all $\lambda,\rho>0$.

\begin{theorem}
\label{thm:lattice}
	Consider the model on $\mathbb{Z}^d$ for $d\geq2$.
	For all $\rho>0$, there exists $\lambda'=\lambda'(\rho,d)<1$ such that if $\lambda>\lambda'$, type 1 dies out almost surely.
\end{theorem}

	The fact that for any $\rho>0$, type 1 dies out almost surely for all $\lambda>1$ is a simple consequence of the classical \emph{shape theorem} for first passage percolation on $\mathbb{Z}^d$ (see Theorem~\ref{thm:shape}).
	Hence, on $\mathbb{Z}^d$, the infection cannot survive forever if awareness spreads faster than the infection.
	Theorem~\ref{thm:lattice} sharpens this result so that for any conversion rate, unaware individuals cannot survive forever if the infection spreads only just faster than awareness.
	
	A complete picture of the survival regimes on $\mathbb{Z}^d$ is an open problem but the expected behaviour is given in the following conjecture.

\begin{conjecture}
\label{conj:lattice}
Consider the model on $\mathbb{Z}^d$ for $d\geq2$.
\begin{itemize}
\item For all $\lambda<1$, if $\rho$ is sufficiently small, type 1 survives with positive probability.
\item There exists $\rho_c=\rho_c(d)\in(0,\infty)$ such that:
\begin{enumerate}
\item if $\rho<\rho_c$ and $\lambda$ is sufficiently small, type 1 survives with positive probability.
\item if $\rho>\rho_c$ and $\lambda>0$, then type 1 dies out almost surely.
\end{enumerate}
\end{itemize}
\end{conjecture}

	The first half of Conjecture~\ref{conj:lattice} is closely related to the strong survival phase for first passage percolation in a hostile environment, studied in Sidoravicius and Stauffer \cite{sidoravicius2019multi} (see Section~\ref{sec:related_work}).
	We believe the proof in our case will be similar to the encapsulation arguments in \cite{sidoravicius2019multi}, though there are some caveats in applying that argument directly.
	
	A priori there is no reason $\rho_c$ is even well-defined in Conjecture~\ref{conj:lattice} due to the non-monotone nature of the model.
	Indeed, several phase transitions may occur.
	We provide a partial answer to Conjecture~\ref{conj:lattice}.
	The following result gives the expected subcritical and supercritical behaviour holds so long as $\rho$ is sufficiently small or sufficiently large, respectively.
	
\begin{theorem}
\label{thm:rho}
Consider the model on $\mathbb{Z}^d$ for $d\geq2$.
\begin{itemize}
\item There exists $\rho_{\ell}=\rho_{\ell}(d)>0$ such that if $\rho<\rho_{\ell}$ and $\lambda$ is sufficiently small, type 1 survives with positive probability.
\item There exists $\rho_{u}=\rho_{u}(d)<\infty$ such that if $\rho>\rho_{u}$, type 1 dies out almost surely for all $\lambda>0$.
\end{itemize}
\end{theorem}

	Proving there exists $\rho_{\ell}>0$ in Theorem ~\ref{thm:rho} relies on a careful coupling with a percolation process analysed by Dauvergne and Sly \cite{dauvergne2021spread}.
	The proof for the existence of $\rho_{u}<\infty$ follows through a direct comparison with Bernoulli percolation.
	Details for both these proofs are given in Section~\ref{sec:lattice}.


\subsection{Related work}	
\label{sec:related_work}

	The model introduced in this paper is closely related to \emph{chase-escape dynamics in the predator-prey model}.
	In the language of our model, the predator-prey dynamics can be viewed in the following manner.
	At time 0, the origin is occupied by type 2 and a neighbour is occupied by type 1. 
	Type 1 spreads at rate $\lambda>0$ to unoccupied sites while type 2 spreads at rate $1$ to both unoccupied sites and sites occupied by type 1.
	Type 1 can be seen as \emph{prey} while type 2 is the \emph{predator}.
	Note we have swapped the role of $\lambda$ to the rate of type 1 to remain consistent with the notation in the predator-prey literature.
	The model described above was introduced by Kordzakhia \cite{kordzakhia2005escape}, who proved a phase transition occurs and computed the exact critical value in the case of the $d$-ary tree.
	More precisely, there exists $\lambda_c>0$ such that if $\lambda>\lambda_c$, type 1 survives with positive probability while if $\lambda<\lambda_c$, type dies out almost surely.
	Bordenave \cite{bordenave2014extinction} extended the results by Kordzakhia to Galton-Watson trees and proved type 1 dies out almost surely at criticality. 
	We also direct the reader to Kortchemski \cite{kortchemski2016predator} and references therein for more information about predator-prey dynamics.

	The question of coexistence in competing random processes on $\mathbb{Z}^d$ has attracted significant attention in recent years.
	The \emph{two-type Richardson model} was introduced by H\"aggstr\"om and Pemantle~\cite{haggstrom1998first} as a model for competing first passage percolation.
	In this model, each type starts from a finite set of sites and spreads at rate $1$ and $\lambda$, respectively. so that once a site is occupied by one process, it remains occupied by that process henceforth. 
	They conjecture coexistence occurs with positive probability if and only if $\lambda=1$, assuming non-degenerate initial conditions.
	See \cite{deijfen2008pleasures} for a background in coexistence regimes in the two-type Richardson model and progress towards this conjecture.

	A more recent competition model that has been studied is \emph{first passage percolation in a hostile environment} (FPPHE), introduced by Sidoravicius and Stauffer \cite{sidoravicius2019multi}.
	The first type in FPPHE initially only occupied the origin while the second is dormant in \emph{seeds} that are distributed as a product of Bernoulli measures of parameter $p$.
	Type 1 spreads at rate 1 through edges and when it encounters a seed, the occupancy is suppressed and the seed is activated.
	Activated seeds then spread at rate $\lambda>0$ to vacant sites.
	Once a site is occupied by either type, it remains that type henceforth.	
	
	Sidoravicius and Stauffer \cite{sidoravicius2019multi} considered FPPHE on $\mathbb{Z}^d$ for $d>1$.
	 They proved for all $\lambda<1$, if $p$ is sufficiently small, then type 1 survives and all components of type 2 are bounded with positive probability - a regime called strong survival.
	 Finn and Stauffer \cite{finn2020non} proved a regime of coexistence exists on $\mathbb{Z}^d$ for $d>2$, in which both types occupy unbounded connected regions.
	 Coexistence on transitive, hyperbolic, non-amenable graphs has also been established in Candellero and Stauffer \cite{candellero2018coexistence}.
	
	FPPHE has been used as an analytical tool in Sidoravicius and Stauffer \cite{sidoravicius2019multi} to analyse a challenging aggregation model called \emph{multiparticle diffusion limited aggregation}. 
	A streamlined version of FPPHE, called Sidoravicius--Stauffer percolation (SSP), was utilised in Dauvergne and Sly \cite{dauvergne2021spread} to study a non-homogeneous spread of infection model (see Section~\ref{sec:SSP}).
	In the above works, coupling the evolution of the desired process with FPPHE or SSP allowed an efficient way to analyze a process with non-equilibrium dynamics without the need to carry out an involved 
	multi-scale analysis from scratch. 	

	We believe our competition process can also be used in this regard. The main idea is that type~1 represents the \emph{propagation front} through ``typically good'' regions of the process being analysed, 
   being it the front of the growth of an aggregate or the 
   front of the propagation of the infection, in the above cases. A type-1 site being converted to type~2 represents that enough time has passed so that that site is not 
   anymore contributing to the propagation front. Being it a random time, a site could have a very short conversion time, not allowing type~1 to spread from it to its neighbors. 
   This models situations where the propagation front may pass through \emph{atypically bad} regions 
   of space-time, which locally blocks the propagation of the front. The spread of type~2 then represents the spread of the influence that bad regions may have on neighboring areas. 
   If one can show that type~1 grows indefinitely despite the expansion of type 2, then such a reasoning would imply that the typical regions are dense enough to compensate the presence of bad regions in space-time, allowing
   the propagation front to survive indefinitely.

\subsection{Overview of paper}

	In Section~\ref{sec:prelim}, we provide a rigorous construction of the model on the $d$-ary tree that will facilitate later proofs and recall some notation from first passage percolation.
	In Section~\ref{sec:d_ary_tree}, we prove Theorem~\ref{thm:crit_lambda_tree} through a careful renormalisation scheme that controls how type 1 and type 2 spreads through the tree.
	In Section~\ref{sec:lattice}, we switch focus to the lattice and prove Theorem~\ref{thm:lattice} and Theorem~\ref{thm:rho}.
	
\section{Preliminaries}
\label{sec:prelim}

\subsection{Construction of the model on the $d$-ary tree}

	In this section we provide a particular construction of the model on the $d$-ary tree that will later be helpful in proofs.

	For $d\geq3$, let $\mathbb{T}_d$ be the $d$-ary tree with a distinguished site called the root, written as $\sigma$. 
	Fix constants $\rho>0$ and $\lambda>0$. 
	Let $V(\mathbb{T}_d)$ and $E(\mathbb{T}_d)$ be the vertex set and edge set of $\mathbb{T}_d$, respectively.
	Let $\{\mathcal{I}_x\}_{x\in V(\mathbb{T}_d)}$ be an i.i.d.\ collection of exponential random variables of rate $\rho$ that correspond to the conversion time of a site.
	That is, given a site $x\in V(\mathbb{T}_d)$, once $x$ is occupied by type 1, after a further time of $\mathcal{I}_x$ has expired, it converts to type 2.
	Let $\{t_{1,e}\}_{e\in E(\mathbb{T}_d)}$ be a collection of i.i.d.\ exponential random variables of rate 1 that correspond to the passage times for type 1. 
	Similarly, let $\{t_{u,e}\}_{e\in E(\mathbb{T}_d)}$ and $\{t_{d,e}\}_{e\in E(\mathbb{T}_d)}$ be two collections of i.i.d.\ exponential random variables of rate $\lambda$ that correspond to the passage times for type 2 upwards and downwards, respectively. 
	That is, if $e=xy$ is an edge with $d(\sigma,x)<d(\sigma,y)$, where $d(\cdot,\cdot)$ is the graph distance metric on $\mathbb{T}_d$, the passage time for type 2 spreading from $x$ to $y$ is given by $t_{d,e}$ and the passage time for type 2 spreading from $y$ to $x$ is given by $t_{u,e}$.
	In this case, we say $y$ is a \emph{descendent} of $x$.
	
	At time $t=0$, the root $\sigma$ is occupied by type 1 and all other sites are vacant. 
	If a site $x$ is occupied by type 1 and $y$ neighbours $x$, then $y$ is attempted to be occupied by type 1 after waiting $t_{1,xy}$ time.
	The occupation is successful if $y$ is vacant and suppressed otherwise. 
	Additionally, once a site is occupied by type 1, after waiting $\mathcal{I}_x$ time, it converts to type 2.
	If a site $x$ is occupied by type 2 and $y$ neighbours $x$, then $y$ is attempted to be occupied by type 2 after waiting $t_{u,xy}$ or $t_{d,xy}$ time, depending on whether $y$ is closer or further than $x$ to the root with respect to the graph-distance metric, respectively. 
	The occupation is successful if $y$ is unoccupied or is occupied by type 1.
	The requirement to distinguish between downward and upward passage times for type 2 is to decouple the different ways type 2 may spread and is an important ingredient in our proofs. 
	
	Given a site $x\in V(G)$, let $\tau_1(x)$ denote the first time that type 1 occupies $x$ and $\tau_2(x)$ denote the first time that type 2 occupies $x$, so that
$$
	\tau_i(x) = \inf\left\{t\geq0: \mbox{type }i \mbox{ occupies }x\right\} \quad \mbox{for }i=1,2.
$$
	If type 1 never occupies the site $x$, we write $\tau_1(x)=\infty$.	
	
\subsection{First passage percolation}	
	
	A \emph{path} is a sequence of distinct sites $\left(x_1,x_2,\ldots,x_n\right)$ such $x_ix_{i+1}\in E(\mathbb{T}_d)$ for each $i\in\left\{1,2,\ldots,n-1\right\}$. 
	Given a path $\gamma=\left(x_1,x_2,\ldots,x_n\right)$, we define the \emph{type \# passage time for }$\gamma$ as the random variable
\begin{equation}
\label{eq:FPP_sum}
	T_{\#}(\gamma)=\sum_{i=1}^{n-1}t_{\#,x_ix_{i+1}},
\end{equation}
where $\#\in\{1,u,d\}$.

	Given $x,y\in V(\mathbb{T}_d)$, define the \emph{type \# passage time from }$x$ \emph{to} $y$ as the random variable
$$
	T_{\#}(x\rightarrow y) = T_{\#}(\gamma_{x,y})
$$
where $\gamma_{x,y}$ is the unique shortest path from $x$ to $y$ and $\#\in\left\{1,u,d\right\}$.
	Note the direction of the arrow in the definition of the passage time will always agree with the descendent structure of the tree.
	That is, if $y$ is a descendent of $x$, the type 2 passage time from $x$ to $y$ is downward, while the type 2 passage time from $y$ to $x$ is upward.
	This is purely a notational subtlety and will not alter any of our arguments.

\section{Survival on the $d$-ary tree}
\label{sec:d_ary_tree}

	In this section we prove Theorem~\ref{thm:crit_lambda_tree}.
	We begin with a roadmap of the proof to help guide the reader in how we establish survival of type 1 with positive probability.
	
\subsection{Roadmap of proof of Theorem~\ref{thm:crit_lambda_tree}}

	Given $d\geq3$, recall $\mathbb{T}_d$ is the $d$-ary tree with $\sigma$ as a distinguished site called the root.
	For any $n\in\mathbb{Z}_{+}$, let $V_n$ denote the set of sites up to graph distance $n$ from the root and $\partial V_n$ denote the set of sites with graph distance exactly $n$ from the origin, so that
$$
	V_n = \left\{x\in V(\mathbb{T}_d):d(\sigma,x)\leq n\right\} \quad \mbox{and} \quad \partial V_n = \left\{x\in V(\mathbb{T}_d):d(\sigma,x)=n\right\},
$$
where $d(\cdot,\cdot)$ is the metric on $\mathbb{T}_d$ induced by shortest-paths between sites.
	We refer to the set $V_n$ as a \emph{box} of depth $n$ from $\sigma$, or just a box for brevity.

	The aim is to partition $\mathbb{T}_d$ into boxes that are labelled as \emph{good} or \emph{bad} according to events measurable with respect to the box, so that a good box implies that type 1 is able to propagate well through the box while type 2 is hindered.
	Let $k$ and $r$ be two large integers we set later. 
	Let $B_{1}=V_{kr}$ and $B_{2}=V_{kr+k^2}\backslash V_{kr}$ be the two segments of our box $B=B_1\cup B_2 = V_{kr+k^2}$.
	The reader should have in mind that we will eventually set $r$ much larger than $k$ so that $kr\gg k^2$.
	
	In $B_1$, the first section of the box, we want to prove that there are sufficiently many \emph{highways} down to $\partial V_{kr}$ with high probability.
	Roughly speaking, highways are paths where type 1 is typically faster than type 2, and we will prove that type 1 survives along these highways with positive probability.
	We aim to prove that many such highways exist in $B_1$ even if $\lambda$ is slightly larger than 1.
	This is made rigorous in Section~\ref{sec:perc}. 
	An illustration of highways in a good box can be seen in Figure~\ref{fig:goodbox}.
	
	Assuming we can show the existence of many highways through $B_1$, we wish to prove the following holds in $B_2$.
	From the set of sites in $\partial V_{kr}$ that are connected to $\sigma$ through highways, there are at least two such sites that can be extended down to depth $kr+k^2$ so that the type 1 passage time on the path is fast and the type 2 passage time on each edge with a site incident to the path is very large.
	The probability a given path satisfies these properties is extremely small, which is precisely why we require many highways existing through $B_1$ to ensure that two such paths exist with high probability.
	These paths that extend highways we refer to as \emph{spines}.
	The reason why we need such a strong property from the spines is that they are very close to the boundary of the box, and we need to ensure that conversions of type 2 from outside the a good box cannot propagate into the box and block type 1.
	These notions are made rigorous in Section~\ref{sec:excellent}.
	An illustration of spines can be seen in Figure~\ref{fig:goodbox}.
	
	We have not yet considered the fact that type 1 converts to type 2 at rate $\rho$. 
	The idea is that we set $\rho$ small enough so that no conversions can take place in a good box until type 1 can spread far down the tree.
	However, conversions could occur outside of a good box, allowing type 2 to spread upwards the tree into a highway or a spine.
	In Section~\ref{sec:backtracks}, we control how type 2 can spread in a good box through conversions occurring outside that good box, so that type 1 is not impeded on the previously constructed highways or spines.
	
	In Section~\ref{sec:good_box}, we put together the previous sections to define a good box and then use this construction of good boxes to prove Theorem~\ref{thm:crit_lambda_tree} via a branching argument in Section~\ref{sec:proof_thm_tree}.	
	The idea is that with positive probability, a branching structure of good boxes gives rise to at least one infinite path of sites such that all sites on this path are occupied by type 1 at some time. 
	
\begin{figure}
  \centering
    \includegraphics[width=0.65\textwidth]{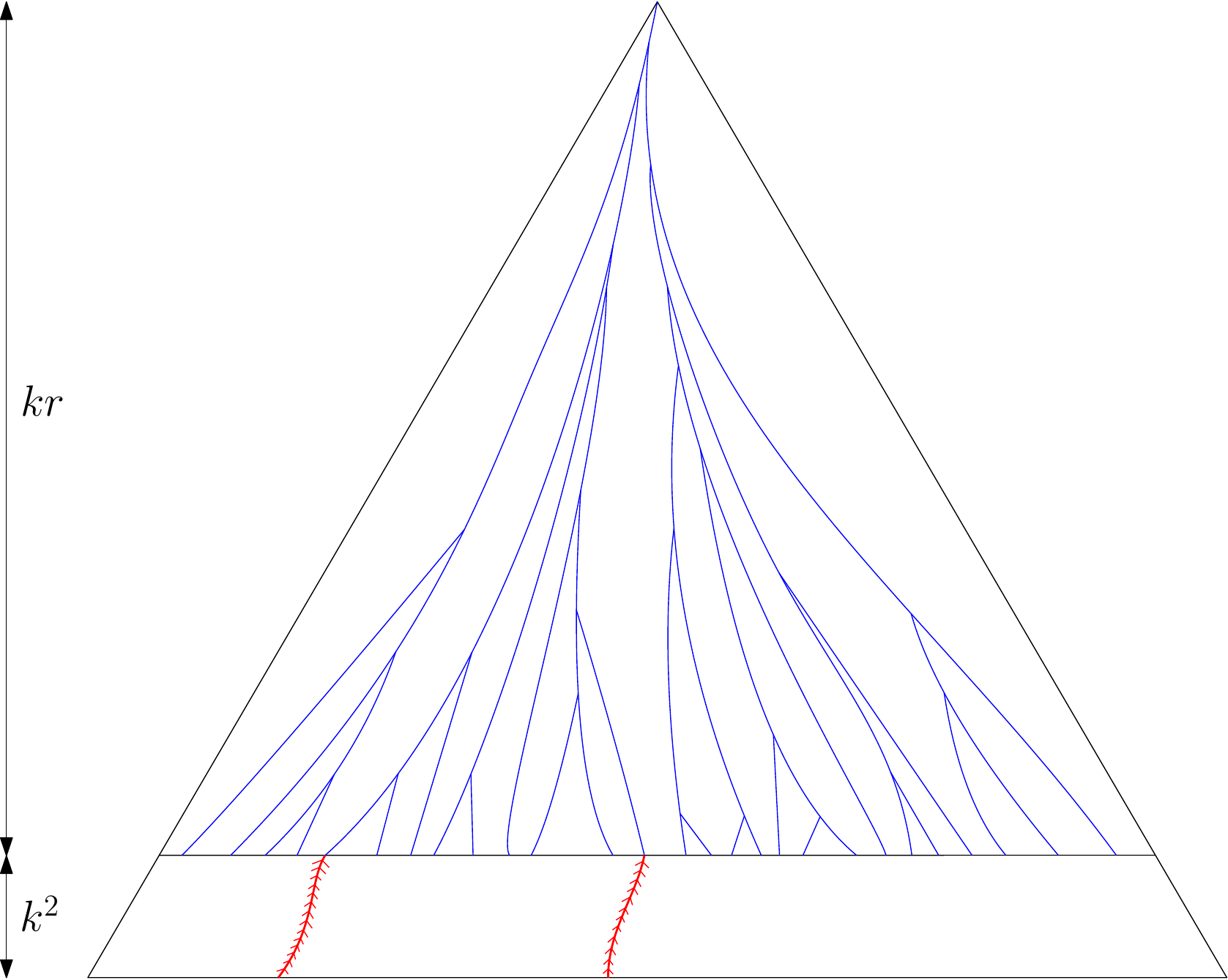}
    \caption{An illustration of a good box split into $B_1$ and $B_2$ with respective depths $kr$ and $k^2$.
    Blue paths represent the highways constructed in $B_1$ that connect the root of the box to depth $kr$. 
    Red paths represent the spines constructed in $B_2$ that extend certain highways to depth $kr+k^2$ in a box as well as any edges incident to them.}
    \label{fig:goodbox}
\end{figure}

\subsection{Percolating structure in $B_1$}
\label{sec:perc}
	
	In this section we prove the existence of a percolating structure that provides the highways needed for type 1 in $B_1$.
	This percolating structure arises from partitioning $B_1$ into \emph{sub-boxes} where we wish to prove there is the existence of paths for which there is strong control for the passage times of type 1 and type 2.
	
	We construct the sub-boxes as follows.
	Let $k$ be a large integer we set later.
	For $z\in V(\mathbb{T}_d)$, let $V^{\downarrow}_k(z)$ (resp.\  $\partial V^{\downarrow}_k(z)$) denote the set of sites of graph-distance less than or equal (resp.\ equal) to $k$ from $z$ such that the path to the root from the site must pass through $z$.

\begin{definition}[Good sub-box]
	Let $z\in V(\mathbb{T}_d)$ and consider the sub-box $V^{\downarrow}_k(z)$.
	For $\varepsilon>0$, let $\mathcal{H}_{\varepsilon}(z)=\mathcal{H}_{\varepsilon}^{(1)}(z)\cap \mathcal{H}_{\varepsilon}^{(\lambda)}(z)$, where
\begin{align}
\label{def:S_eps}
	& \mathcal{H}_{\varepsilon}^{(1)}(z)=\left\{y\in\partial V^{\downarrow}_k(z):T_1(z\rightarrow y)\leq (1-\varepsilon)k\right\},\\
& \mathcal{H}_{\varepsilon}^{(\lambda)}(z)=\left\{y\in\partial V^{\downarrow}_k(z):T_d(z\rightarrow y)\geq(1-\varepsilon^2)\tfrac{k}{\lambda}\right\}. \label{def:S_eps2}
\end{align}

	Define $V^{\downarrow}_k(z)$ to be $\varepsilon$-\emph{good} if $\mathcal{H}_{\varepsilon}(z)$ contains at least two distinct elements and $\varepsilon$-\emph{bad} otherwise.
	
	Given $y\in\mathcal{H}_{\varepsilon}(z)$, the path from $z$ to $y$ is called a \emph{highway}.
\end{definition}

\begin{remark}
	Note that in \eqref{def:S_eps2} in the definition of $ \mathcal{H}_{\varepsilon}^{(\lambda)}(z)$, we are using the \emph{downward} type 2 passage times.
\end{remark}
	
	With the notion of good sub-boxes to hand, we wish to prove that there is a percolating structure of good sub-boxes up to depth $kr$.	
	We first partition $B_1$ into sub-boxes of depth $k$, so that each sub-box is independently $\varepsilon$-good of every other sub-box.
	That is, the sub-boxes are of the form 
$$
	V_k^{\downarrow}(z) \mbox{ for }z\in\partial V_{ik} \mbox{ with } i\in\{0,1,2,....,r-1\}.
$$ 
	A sub-box being good means that there are at least two paths from its root to depth $k$ where there is a strong control on type 1 and type 2. 
	In particular, if $\lambda<1+\varepsilon$, then type 1 is faster along these paths since
$$
	1-\varepsilon = \frac{1-\varepsilon^2}{1+\varepsilon} < \frac{1-\varepsilon^2}{\lambda}.
$$ 
	In Figure~\ref{fig:subbox}, we see how highways in good sub-boxes can join, giving long highways that allow for good control on type 1 and type 2.

	To ensure that many such highways exist down to depth $kr$, we need to prove that sub-boxes are good with high probability. 
	This is the content of the following lemma.
	
\begin{figure}
  \centering
    \includegraphics[width=0.3\textwidth]{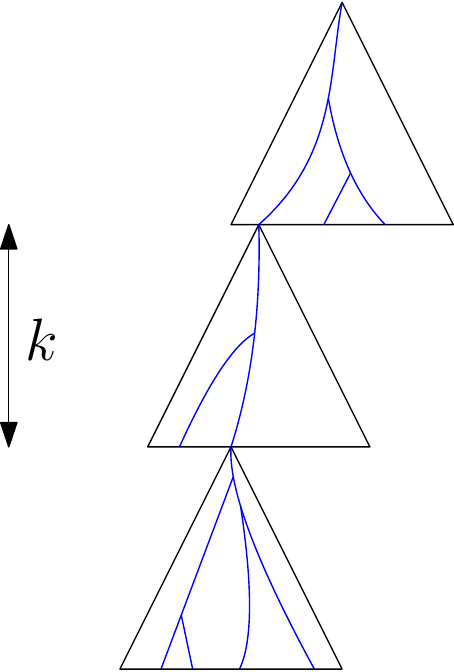}
    \caption{Triangles represent good sub-boxes and blue paths represent the highways constructed within them.}
    \label{fig:subbox}
\end{figure}
	
\begin{lemma}
\label{lem:prob_1box}
	There exists a constant $c>0$ such that the following holds.
	Fix $z\in V(\mathbb{T}_d)$.
	For all sufficiently small $\varepsilon>0$ and large enough $k$,
$$
	\mathbb{P}\left(V^{\downarrow}_k(z)\mbox{ is }\varepsilon\mbox{-good}\right)\geq 1-\exp\left(-c\varepsilon^4 k\right).
$$
\end{lemma}

	To prove Lemma~\ref{lem:prob_1box}, we recall some results from the theory of branching random walks, that can be viewed as an alternative representation of first passage percolation on trees. 
	The following result is due to Addario-Berry and Reed \cite[Theorem 3]{addario2009minima}, although we emphasise their result holds in a much greater generality and we state it only in what will be useful in our context.

\begin{theorem}
\label{thm:BRW}
	For $n\geq1$, let $M_n$ be the infimum over all passage time from $\sigma$ to $\partial V_n$, so that
$$
	M_n = \inf_{y\in \partial V_n}T_1(o\to y).
$$ 
	There exists a constant $\varepsilon_0\in(0,1)$ such that 
$$
	\mathbb{E}\left[ M_n \right] \leq (1-\varepsilon_0)n
$$
 for all $n$ large enough.
  Moreover, there exist constants $C,\delta>0$ such that for all $x\in\mathbb{R}$,
$$
	\mathbb{P}\left(\left|M_n-\mathbb{E}\left[M_n\right]\right|\geq x\right)\leq Ce^{-\delta x}.
$$
\end{theorem}

	The result in Theorem~\ref{thm:BRW} concerns only the fastest path, but we will require that many paths can satisfy its conditions.
	This is the content of the following lemma. 
	Recall the definition of $\mathcal{H}_{\varepsilon}^{(1)}(z)$ in \eqref{def:S_eps}.

\begin{lemma}
\label{lem:fast_type1}
	There exists a constant $c>0$ such that the following holds.
	Fix $z\in V(\mathbb{T}_d)$.
	For all $\varepsilon>0$ small enough, there exists $k_1=k_1(\varepsilon)$ such that if $k>k_1$, then
$$
	\mathbb{P}\left(\left|\mathcal{H}_{\varepsilon}^{(1)}(z)\right|\geq d^{\varepsilon k}\right)\geq 1 - e^{-c\varepsilon k}.
$$
\end{lemma}

\begin{proof}
	Firstly, we want to prove that the passage time over all paths from $z$ to $\partial V_{\varepsilon k}^{\downarrow}(z)$ is bounded above by $C\varepsilon k$ with sufficiently high probability, where $C$ is a large constant and $\varepsilon$ is a small constant we set later.
	For a given site $y\in\partial V_{\varepsilon k}^{\downarrow}(z)$ and a large enough constant $C$,  we see through a Chernoff bound argument (see Lemma~\ref{lem:chernoff}) that
$$
	\mathbb{P}\left(T_1(z\to y)\geq C\varepsilon k\right)\leq e^{-\tfrac{C\varepsilon k}{2}}.
$$
By taking the union bound over all $y\in\partial V_{\varepsilon k}^{\downarrow}(z)$,
$$
	\mathbb{P}\left(\sup_{y\in\partial V_{\varepsilon k}^{\downarrow}(z)}T_1(\sigma\to y)\geq C\varepsilon k\right)\leq d^{\varepsilon k}e^{-\tfrac{C\varepsilon k}{2}}.
$$
	
	 By Theorem~\ref{thm:BRW}, there exists a constant $\delta>0$, that does not depend on $k$, such that for each $y\in\partial V_{\varepsilon k}^{\downarrow}(z)$,  the following holds. 
	For all $x>0$
$$
	\mathbb{P}\left(T_1(y\to\partial V_k^{\downarrow}(z))\geq (1-\varepsilon_0)(1-\varepsilon)k+x\right)\leq e^{-\delta x},
$$
where $\varepsilon_0$ is as given in Theorem~\ref{thm:BRW}.
	By the union bound, for all $x>0$
$$
	\mathbb{P}\left(\bigcup_{y\in\partial V_{\varepsilon k}^{\downarrow}(z)}\left\{T_1(y\to\partial V^{\downarrow}_k(z))\geq (1-\varepsilon_0)(1-\varepsilon)k+x\right\}\right) \leq d^{\varepsilon k}e^{-\delta x}.
$$
If $x = 2\varepsilon k\left(\log d\right) /\delta$, then the right-hand term above can be bounded above by $e^{-\delta x/2}$.
	Hence
\begin{align*}
	\mathbb{P}\Bigg(&\bigcap_{y\in\partial V_{\varepsilon k}^{\downarrow}(z)} \left\{T_1(y\to\partial V_k^{\downarrow}(z))< (1-\varepsilon_0)(1-\varepsilon)k+\tfrac{2\varepsilon k\log d}{\delta}\right\}\cap \left\{\sup_{y\in\partial V_{\varepsilon k}^{\downarrow}(z)}T_1(z\to y)< C\varepsilon k\right\}\Bigg)\\
	 & \geq 1-e^{-\tfrac{\delta\varepsilon k}{4}}-e^{-\tfrac{C \varepsilon k}{4}},\\
& \geq 1-e^{-c\varepsilon k},
\end{align*}
for some constant $c>0$ and all sufficiently large $k$.
	Moreover, the event directly above implies the existence of at least $d^{\varepsilon k}$ sites in $\partial V_{k}^{\downarrow}(z)$, so that if $y$ is one such site,
\begin{align*}
T_1(z\to y)\leq (1-\varepsilon_0)(1-\varepsilon)k+\tfrac{2\varepsilon k\log d}{\delta} +C\varepsilon k\leq (1-\varepsilon)k,
\end{align*}
where the final inequality holds so long as $\varepsilon$ satisfies
$$
	\varepsilon \leq \frac{\varepsilon_0}{\varepsilon_0+C+\tfrac{2\log d}{\delta}}.
$$
\end{proof}

	With Lemma~\ref{lem:fast_type1} established, we are now in a position to prove Lemma~\ref{lem:prob_1box}.

\begin{proof}[Proof of Lemma~\ref{lem:prob_1box}]
	By Lemma~\ref{lem:fast_type1}, there exists a constant $c>0$ such that for a small enough choice of $\varepsilon>0$ and then large enough choice of $k$,  
\begin{equation}
\label{pf:lm32_1}
	\mathbb{P}\left(\left|\mathcal{H}_{\varepsilon}^{(1)}(z)\right|\geq2\right)\geq 1-e^{-c\varepsilon k}.
\end{equation}

	Conditional on the event $\{|\mathcal{H}_{\varepsilon}^{(1)}(z)|\geq2\}$, let $y_1,y_2\in \mathcal{H}_{\varepsilon}^{(1)}(z)$ be distinct sites.
	For $i\in\{1,2\}$, let $E_i$ be the event
$$
	E_i=\left\{T_d(z\rightarrow y_i) \geq \left(1-\varepsilon^2\right)\tfrac{k}{\lambda}\right\}.
$$
	Since $T_d(z \to y_i)$ is a sum of $k$ i.i.d.\ Exponential($\lambda$) random variables, by a Chernoff bound argument (see Lemma~\ref{lem:chernoff}), we deduce there exists a constant $c'>0$ such that for all sufficiently large $k$,
$$
	\mathbb{P}\left(E_i\right)\geq 1-\exp\left(-c'\varepsilon^4 k\right) \quad \mbox{for }i=1,2.
$$
	By the union bound, 
\begin{equation}
\label{pf:lm32_2}
	\mathbb{P}\left(E_1\cap E_2\right)\geq 1 - \mathbb{P}(E_1^c) - \mathbb{P}(E_2^c) \geq 1 - 2\exp\left(-c'\varepsilon^4 k\right).
\end{equation}
	The result follows by \eqref{pf:lm32_1} and \eqref{pf:lm32_2} through the independence of the passage times for type 1 and type 2.
\end{proof}

	To ease the statements of results, henceforth we assume that $\varepsilon$ is small enough and $k$ is large enough so Lemma~\ref{lem:prob_1box} is satisfied.
	
	Consider the box $B=V_{kr+k^2}$.
	Let $\mathcal{P}_k\subset\partial V_{kr}$ be the set of sites whose geodesic to $\sigma$ only passes through highways of good sub-boxes. 
	Our proof of Theorem~\ref{thm:crit_lambda_tree} will rely on $\mathcal{P}_k$ containing sufficiently many elements, which is established in the following lemma.
	
\begin{lemma}
\label{lem:depth_kr}
	Fix $\alpha\in(1,2)$.
	If there exists a constant $c$ such that $r\leq k^c$, then
$$
	\lim_{k\to\infty}\mathbb{P}\left(\left|\mathcal{P}_k\right|>\alpha^r\right)= 1.
$$
\end{lemma}	
	
	The proof of Lemma~\ref{lem:depth_kr} is a consequence of Lemma~\ref{lem:prob_1box} and the following elementary result, as the sub-boxes in $B_1$ are independently good or bad.

\begin{lemma}
	Fix $\alpha\in(1,2)$ and $r\in\mathbb{N}$.
	Consider independent site percolation of parameter $p=p(r)$ on the binary tree so that
$$
	\lim_{r\to\infty}p^r = 1.
$$
	Let $N_r$ be the number of sites at depth $r$ connected to the origin by an open path.
	Then
$$
	\lim_{r\to\infty}\mathbb{P}\left(N_r > \alpha^r\right)= 1.
$$
\end{lemma}

\begin{proof}
By applying Markov's inequality to the number of sites not connected to the root by open sites, we deduce
$$
	\mathbb{P}\left(2^r-N_r\geq 2^r-\alpha^r\right) \leq \frac{2^r-\mathbb{E}[N_r]}{2^r-\alpha^r} = \frac{1-p^r}{1-(\alpha/2)^r},
$$
and the result follows.
\end{proof}

\subsection{Construction of spines}
\label{sec:excellent}
	
	In Lemma~\ref{lem:depth_kr} we proved in the box $B=V_{kr+k^2}$, there exists many sites at depth $kr$ that are connected to the root of $B$ through highways with high probability.
	The aim of this section is to extend some of these highways down to depth $kr+k^2$ through what we will refer to as \emph{spines}, so that spines provide a strong control on the type 1 and type 2 passage times.
	
	Recall $\mathcal{P}_k$ is the set of sites $y$ in $\partial V_{kr}$ such that the path from $\sigma$ to $y$ only passes through the highways in good sub-boxes in $B_1$.

\begin{definition}[Spine]
	For each $z\in \mathcal{P}_k$, let $s(z)\in\partial V^{\downarrow}_{k^2}(z)$ satisfy
$$
	T_1(z\to s(z)) = \inf_{y\in\partial V^{\downarrow}_{k^2}(z)}T_1(z\rightarrow y).
$$
The path $s_{z}$ from $z$ to $s(z)$ is defined to be the \emph{spine from }$z$.
\end{definition}

	Given $z\in\mathcal{P}_k$, we want $s_z$ to satisfy the following properties. 
	Firstly, we want the type 1 passage time on $s_z$ to be bounded above by $(1-\varepsilon)k^2$, so type 1 can readily spread to depth $kr+k^2$.
	Secondly, we want every edge incident to a site in the spine $s_z$ to have a type 2 passage time of at least $k^3$ with the exception of edges only incident to $s(z)$. 
	The edges with both endpoints contained on $s_z$ will be measured with the downward passage times and edges with only one endpoint on $s_z$ will be measured with the upward passage time.
	This will give the required impediment to stop type 2 from blocking type 1 on highways.
	Let $\mathcal{S}_k\subset\mathcal{P}_k$ be the set of sites that satisfy these properties, so that $\mathcal{S}_k=\mathcal{S}_k^{(1)}\cap\mathcal{S}_k^{(2)}\cap\mathcal{S}_k^{(3)}$, where
\begin{align*}
\mathcal{S}_k^{(1)} &= \left\{z\in\mathcal{P}_k:T_1(s_z)\leq (1-\varepsilon)k^2\right\},\\
\mathcal{S}_k^{(2)} &= \left\{z\in\mathcal{P}_k: \bigcap_{y\in s_z\backslash s(z)}\bigcap\limits_{\substack{y'\sim y\\ y'\notin s_z}}\left\{T_u(yy')\geq k^3\right\}\right\},\\
\mathcal{S}_k^{(3)} &= \left\{z\in\mathcal{P}_k: \bigcap\limits_{\substack{y,y'\in s_z \\ y\sim y'}}\left\{T_d(yy')\geq k^3\right\}\right\}.
\end{align*}
	It is clear the probability that a given site $z\in\mathcal{P}_k$ satisfies the properties above is extremely small.
	We will recover that $|\mathcal{S}_k|\geq2$ with high probability so long as $\mathcal{P}_k$ contains sufficiently many elements.
	This can be achieved by setting $r=k^6$, and henceforth, $r$ shall take this value.
	This is established formally in the following lemma.

\begin{lemma}
\label{lem:good_spine}
	Let $r=k^6$ and assume $|\mathcal{P}_{k}|> \alpha^r$ for some $\alpha\in(1,2)$. 
	Then,
$$
	\lim_{k\to\infty}\mathbb{P}\left(|\mathcal{S}_k|\geq2\right)=1.
$$
\end{lemma}

\begin{proof}
	The collection of random variables $\left\{T_1(s_z)\right\}_{z\in \mathcal{P}_k}$ are independent, and by Theorem~\ref{thm:BRW}, there exists a constant $\delta>0$ that does not depend on $k$, such that for each $z\in\mathcal{P}_k$,
$$
	\mathbb{P}\left(T_1(s_z)\geq (1-\varepsilon)k^2\right)<e^{-\delta k^2}.
$$
	Moreover, since the probability that the type 2 passage time of an edge is at least $k^3$ is $e^{-\lambda k^3}$, by independence of type 1 and type 2 passage times, for each $z\in\mathcal{P}_k$,
\begin{align}
	\mathbb{P}&\left(\left\{T_1(s_z)< (1-\varepsilon)k^2\right\}\cap\bigcap_{y\in s_z\backslash s(z)}\bigcap\limits_{\substack{y'\sim y\\ y'\notin s_z}}\left\{T_u(yy')\geq k^3\right\}\cap\bigcap\limits_{\substack{y,y'\in s_z \\ y\sim y'}}\left\{T_d(yy')\geq k^3\right\}\right) \nonumber\\
	&> \left(1 - e^{-\delta k^2}\right)e^{-\lambda k^3\cdot (d-1) k^2} \nonumber\\
	 &\geq e^{-c k^5},  \label{proof:cand_paths}
\end{align}
for some constant $c=c(\lambda)>0$.
	From \eqref{proof:cand_paths} we deduce the probability there exist at least two such paths tends to 1 as $k$ goes to infinity as there are more than $\alpha^{k^6}$ independent candidate paths.
\end{proof}

\subsection{Controlling type 2 from conversions}
\label{sec:backtracks}

	In this section we consider how to control the spread of type 2 from conversions.
	The idea is that we may set $\rho$ so small so that in good boxes, no conversions take place until type 1 can spread far down the tree with high probability.
	However, we must also control for conversions that occur outside of good boxes.
	For example, type 1 may spread fast through a good box and trigger an instantaneous conversion just outside it, that in turn causes the spread of type 2.
	This spread may go back through the good box and prevent highways from being occupied by type 1.
	The aim of this section is to control how type 2 spreads from conversions outside of good boxes in such a manner that type 1 is not impeded on highways.
	
	Consider a box $B$ and suppose $|\mathcal{P}_k|>\alpha^r$ for some $\alpha\in(1,2)$ and $|\mathcal{S}_k|\geq2$, so that $\pi_1$ and $\pi_2$ are highways of $B$ with respective spines $s_1$ and $s_2$. 
	If more highways and spines are available, we simply ignore them.	
	The idea is that type 2 must traverse a distance of at least $k^2$ to convert a site outside $B$ and then occupy a site in $\pi_1\cup\pi_2$, and so by the time type 1 occupies some $x\in\pi_1\cup\pi_2$, then it is able to propagate downwards before being occupied by type 2. 
	We make this notion rigorous below.
	
	For $x\in\pi_1\cup\pi_2$, let $D(x)$ be the set of sites up to distance $k^2$ down from $x$ that exclude all sites whose path to $x$ include an edge on the highways $\pi_1$ and $\pi_2$ or the spines $s_1$ and $s_2$ (call this collection of excluded sites $H_{k^2}(x;\pi_1,\pi_2,s_1,s_2)$), so
$$
	D(x) = V^{\downarrow}_{k^2}(x)\backslash H_{k^2}(x;\pi_1,\pi_2,s_1,s_2).
$$
	Let $\partial D(x) = \partial V_{k^2}(x)\cap D(x)$ and $D^*(x)$ be the event that all paths through $D(x)$ \emph{upwards} have passage time at least $10k$, so that
\begin{equation}
\label{eq:D_star}
	D^*(x) = \bigcap_{y\in \partial D(x)}\left\{T_u(y\to x)\geq 10k\right\}.
\end{equation}
	Intuitively, the event $D^*(x)$ guarantees that no upwards type 2 passage time through $D(x)$ is fast enough to allow type 2 to catch type 1 on $\pi_1\cup\pi_2$.
	It is important the event in \eqref{eq:D_star} is independent of the labelling sub-boxes as good or bad, as a good box is defined through the type 1 passage times and the \emph{downward} type 2 passage times.
	
	The following lemma allows us to prove that $D^*(x)$ occurs for all $x\in\pi_1\cup\pi_2$ with high probability, assuming the existence of highways $\pi_1,\pi_2$ and  spines $s_1,s_2$.
	
\begin{lemma}
\label{lem:backtracks}
	Let $r=k^6$.
	Assume $|\mathcal{P}_k|> \alpha^r$ for some $\alpha\in(1,2)$ and $|\mathcal{S}_k|\geq2$.
	Let $\pi_1$ and $\pi_2$ be highways with respective spines $s_1$ and $s_2$. 
	There exists a constant $c>0$ such that
$$
\mathbb{P}\left(\bigcap_{x\in\pi_1\cup\pi_2}D^*(x)\right)\geq 1 - e^{-ck^2\log k},
$$
for all $k$ large enough.
\end{lemma}

\begin{proof}
	If $\gamma$ is a path of length $k^2$ and $\theta>0$, then
$$
	\mathbb{P}\left(T_u(\gamma)\leq 10k\right) = \mathbb{P}\left(\sum_{i=1}^{k^2}T_u(e_i)\leq 10k\right) \leq e^{10\theta k}\mathbb{E}\left[e^{-\theta T_u}\right]^{k^2},
$$
by a Chernoff bound argument, where $\left\{e_i\right\}_i$ is an enumeration of the edges on the path $\gamma$ and $T_u$ is a copy of an exponential random variable of rate $\lambda$. 
	By setting $\theta = k/10$, we deduce
$$
	\mathbb{P}\left(T_u(\gamma)\leq 10k\right) \leq \left(\tfrac{\lambda}{\lambda+k/10}\right)^{k^2}e^{k^2} = e^{-ck^2\log k},
$$
for some constant $c>0$.
	There are at most $d^{k^2}$ sites in $\partial D(x)$ and hence,  by the union bound,
$$	
	\mathbb{P}\left(\bigcup_{y\in\partial D(x)} \left\{T_u(y\to x)\leq 10k\right\}\right) \leq d^{k^2}e^{-ck^2\log k} \leq e^{-\tfrac{c}{2}k^2\log k},
$$
for large enough $k$.
	Recall the depth of the percolating structure in a box is $kr$ where $r=k^6$.
	By the union bound over all sites in $\pi_1\cup\pi_2$, 
$$
	\mathbb{P}\left(\bigcup_{x\in\pi_1\cup\pi_2}\bigcup_{y\in\partial D(x)} \left\{T_u(y\to x)\leq 10k\right\}\right) \leq 2k^7e^{-\tfrac{c}{2}k^2\log k},
$$
for all large enough $k$ and the result follows.
\end{proof}

\subsection{Good boxes}
\label{sec:good_box}

	In this section we put together the components constructed in the previous sections to rigorously define good boxes. 
	While the events discussed in previous sections only concerned the box containing the origin, the events can easily be translated to an arbitrary box when we partition $\mathbb{T}_d
$ into boxes. 
	Indeed, as we will see in this section, a box being good or bad only depends on events measurable with respect to said box. 
	Hence we will retain the notation from previous sections without introducing ambiguity.
	
	Fix a box $B$ and consider the following events. 
	The first two events concern the construction of highways down to depth $kr$, where we fix $r=k^6$ and $\alpha\in(1,2)$, and at least two spines, so that
$$
	\mathcal{G}_1= \left\{\left|\mathcal{P}_k\right|>\alpha^r \right\} \quad \mbox{and} \quad \mathcal{G}_2 = \left\{\left|\mathcal{S}_k\right|\geq2\right\}.
$$
	Assuming that $\mathcal{G}_1$ and $\mathcal{G}_2$ hold, let $\pi_1$ and $\pi_2$ be two highways to depth $kr$.
	We let $\mathcal{G}_3$ be the event that all sites on $\pi_1\cup\pi_2$ satisfy the condition in \eqref{eq:D_star}, so that
$$
	\mathcal{G}_3 = \bigcap_{x\in\pi_1\cup\pi_2}D^*(x).
$$
	The final ingredient is the event that no conversion occurs in the box until at least time $3(kr+k^2)$ has expired.
	This guarantees that type 1 is able to traverse a large distance before type 2 originating from that box is able to spread.
	Call this event $\mathcal{G}_4$, so that
$$
	\mathcal{G}_4 = \bigcap_{x\in B}\left\{\mathcal{I}_x\geq 3(kr+k^2)\right\},
$$
where we recall $\mathcal{I}_x$ is the conversion time to type 2 at site $x$.

	We define a box $B$ to be \emph{good} if $\cap_{j=1}^{4}\mathcal{G}_j$ holds, and \emph{bad} otherwise.

\begin{lemma}
\label{lem:prob_good}
We have
$$
	\lim_{k\to\infty}\lim_{\rho\downarrow 0}\mathbb{P}\left(B\emph{ is good}\right) = 1.
$$
\end{lemma}

\begin{proof}
	As the conversion times are independent of the passage times, we may express the probability of a box being good as
$$
	\mathbb{P}\left(B\mbox{ is good}\right) = \mathbb{P}\left(\cap_{j=1}^{4}\mathcal{G}_j\right) = \mathbb{P}\left(\mathcal{G}_1\right)\mathbb{P}\left(\mathcal{G}_2|\mathcal{G}_1\right)\mathbb{P}\left(\mathcal{G}_3|\mathcal{G}_1\cap\mathcal{G}_2\right)\mathbb{P}\left(\mathcal{G}_4\right).
$$
	As the event $\mathcal{G}_3$ only observes the upwards type 2 passage times, then it is independent of $\mathcal{G}_1\cap\mathcal{G}_2$. 
	The result then follows by Lemma~\ref{lem:depth_kr}, Lemma~\ref{lem:good_spine} and Lemma~\ref{lem:backtracks}.

\end{proof}

\subsection{Proof of Theorem~\ref{thm:crit_lambda_tree}}
\label{sec:proof_thm_tree}

	In this section we prove Theorem~\ref{thm:crit_lambda_tree}.
	Before proceeding with the proof, we introduce the following notion that will ease exposition. 
	If a site $z$ is occupied by type 2 because it converts to type 2 due to its own conversion time expiring after being occupied by type 1, then we say $z$ is its own \emph{progenitor}. 
	Otherwise, $z$ was occupied by type 2 by the spread of type 2 from a neighbouring site, $z'$ say. 
	If $z'$ is its own progenitor, then we say it is the progenitor for $z$ too. 
	If $z'$ is not its own progenitor, then we can iterate this procedure until we find a site that is the progenitor for $z$ and itself. 
	In general, we write $p(z)$ for the progenitor for $z$ and $p_z$ as the path from $p(z)$ to $z$.
	Note a progenitor must be occupied by type 1 at some time due to being occupied by type 2 through its own conversion.
	As the passage times are exponentially distributed, the progenitor is almost surely unique.

\begin{proof}[Proof of Theorem~\ref{thm:crit_lambda_tree}]
	Let $\varepsilon$ be small enough and $k$ large enough so that Lemma~\ref{lem:prob_1box} holds, and fix $\lambda<1+\varepsilon$.
	We construct an infinite path from $\sigma$ that only passes through highways and spines of good boxes with positive probability as follows.
	Consider the box up to depth $kr+k^2$ with root $\sigma$.
	If this box is good, there exists at least two sites at depth $kr+k^2$ that are connected to the root $\sigma$ via highways and spines of this box.
	Considering these sites as roots of boxes to depth $2(kr+k^2)$, by observing if these boxes are good or bad, we deduce a lower bound on the number of sites connected to $\sigma$ only through highways and spines of good boxes.
	We can continue this procedure $n$ times to attempt to find sites at depth $n(kr+k^2)$ from the origin that only pass through the highways and spines of good boxes.
	As boxes constructed this way are independently good or bad, if we set $k$ large enough and then $\rho$ small enough so that
\begin{equation}
\label{eq:perc_prob}
	\mathbb{P}\left(V_{kr+k^2}\mbox{ is good}\right)>1/2,
\end{equation}
this procedure does not terminate with positive probability.
	If this procedure does not terminate, there exists an infinite path $\gamma$ from $\sigma$ that only passes through the highways and spines of good boxes.
	We aim to prove the existence of $\gamma$ implies type 1 survives, which would complete the proof.
	
	Suppose for contradiction there exists an infinite path from $\sigma$ that only passes through the highways and spines of good boxes, $\gamma$ say, and type 1 does not occupy every site on $\gamma$.
	Then there must exist some good box $B_0$ such that $\gamma$ contains a highway and spine of $B_0$,  type 1 occupies some site in $B_0\cap\gamma$ but not every site in $B_0\cap\gamma$.
	
	Write $B_0\cap\gamma=\pi\cup s$ for the respective highway $\pi$ and spine $s$ of $B_0$, and let $\sigma_{B_0}$ denote the root of $B_0$.
	Suppose there exists $z\in\pi\cup s$ that is never occupied by type 1.
	Without losing generality, we may assume $z$ is the closest such site to $\sigma_{B_0}$.
	Recall $p(z)$ is the progenitor of $z$ and $p_z$ is the path from $p(z)$ to $z$. 
	
	First consider the case $p(z)\in B_0$ (see Figure~\ref{fig:progin} for an illustration).
	Then 
\begin{equation}
\label{eq:cont_conv}
	\tau_2(z) > \tau_1(p(z)) + 3(kr+k^2) > \tau_1(\sigma_{B_0}) + 3(kr+k^2),
\end{equation}
where in the first inequality we recall $B_0$ is good and the second follows as type 1 must occupy $\sigma_{B_0}$ before $p(z)$ by construction.
	As $B_0$ is good, then
\begin{equation}
\label{eq:cont_conv2}
	T_1(\sigma_{B_0}\to z) < (1-\varepsilon)(kr+k^2).
\end{equation}

\begin{figure}
     \centering
     \begin{subfigure}[b]{0.3\textwidth}
         \centering
         \includegraphics[width=\textwidth]{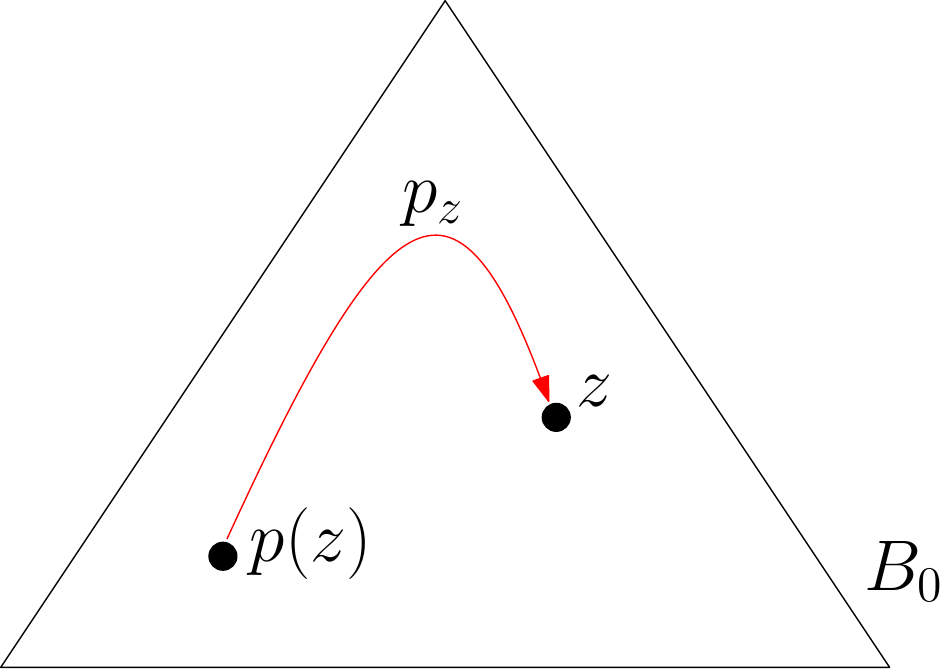}
         \caption{$p(z)\in B_0$}
         \label{fig:progin}
     \end{subfigure}
     \hfill
     \begin{subfigure}[b]{0.3\textwidth}
         \centering
         \includegraphics[width=\textwidth]{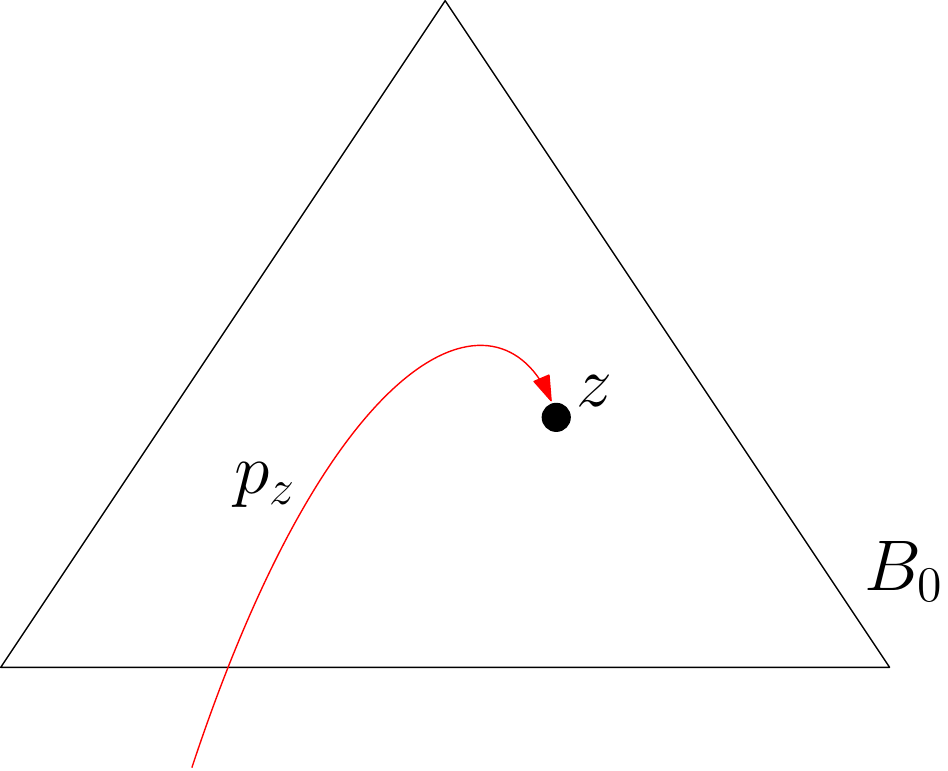}
         \caption{$p_z\cap \partial V_{kr+k^2}^{\downarrow}(\sigma_{B_0})\neq\emptyset$}
         \label{fig:progout1}
     \end{subfigure}
     \hfill
     \begin{subfigure}[b]{0.3\textwidth}
         \centering
         \includegraphics[width=\textwidth]{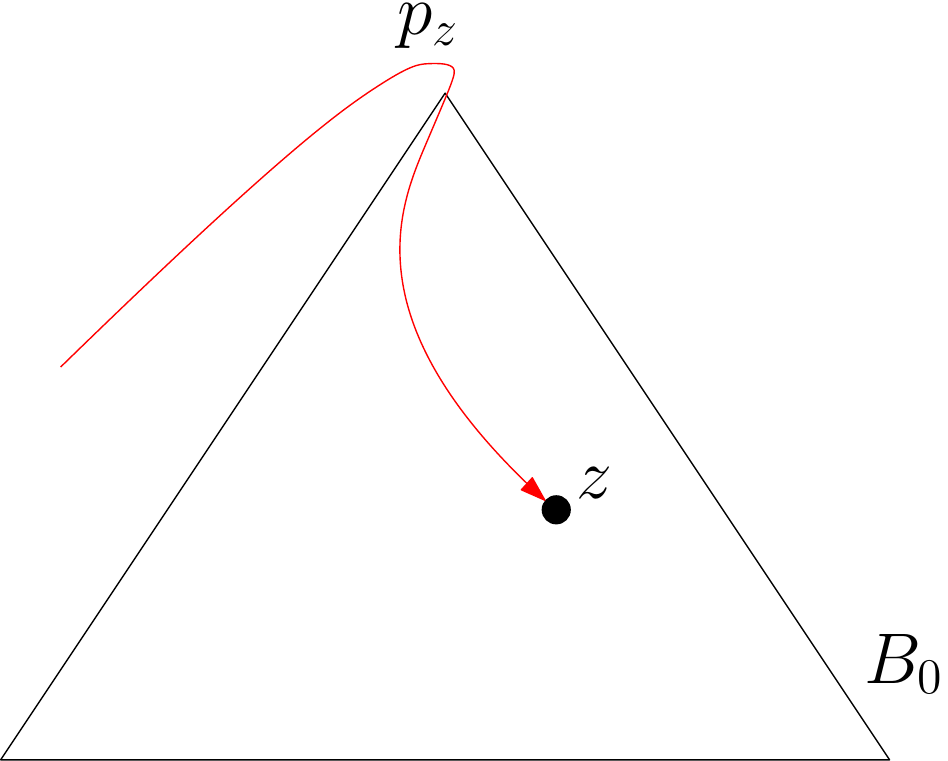}
         \caption{\centering $p_z\cap \partial V_{kr+k^2}^{\downarrow}(\sigma_{B_0})=\emptyset$ and $p(z)\notin B_0$}
         \label{fig:progout2}
     \end{subfigure}
        \caption{The three cases for the nature of $p_z$, the path from $p(z)$ to $z$. The red line is the path $p_z$. The top of the triangle is the root $\sigma_{B_0}$ and the base represents the set $\partial V_{kr+k^2}^{\downarrow}(\sigma_{B_0})$.}
        \label{fig:progcase}
\end{figure}

	Comparing \eqref{eq:cont_conv} and \eqref{eq:cont_conv2}, we deduce  
$$
	\tau_1(z) < \tau_1(\sigma_{B_0}) + (1-\varepsilon)(kr+k^2) < \tau_2(z),
$$
which contradicts the assumption $z$ is never occupied by type 1.
	Hence $p(z)\notin B_0$.
	
	Now consider the case $p_z\cap \partial V_{kr+k^2}^{\downarrow}(\sigma_{B_0})\neq\emptyset$ (see Figure~\ref{fig:progout1}).
	Let $z'\in\pi\cup s$ be such that 
$$
	d(\sigma_{B_0},z') = \min_{y\in B_0\cap\gamma\cap p_z}d(\sigma_{B_0},y).
$$
	It is immediate $d(\sigma_{B_0},z')\leq d(\sigma_{B_0},z)$ through the construction of the progenitor.
	
	First consider the case $z\in\pi$.
	Then
$$
	\tau_2(z) > \tau_2(z') + \tfrac{(1-\varepsilon^2)k}{\lambda}\left(\text{SB}(z,z')-2\right)^+,
$$
where SB$(z,z')$ is the number of sub-boxes that intersect the path from $z'$ to $z$ and for $c\in\mathbb{R}$, we write $(c)^+=\max\left\{c,0\right\}$.
	Recalling $\lambda<1+\varepsilon$, \eqref{eq:D_star} and that $B_0$ is good, we deduce
\begin{equation}
\label{eq:cont_below_pi}
	\tau_2(z) > \tau_1(z') + 10k +(1-\varepsilon)\left(\text{SB}(z,z')-2\right)^+k.
\end{equation}
	Similarly, recalling $\pi$ is a highway and $B_0$ is good,  we deduce
\begin{equation}
\label{eq:cont_below_pi1}
	T_1(z'\to z) < (1-\varepsilon)\text{SB}(z,z')k.
\end{equation}
By comparing \eqref{eq:cont_below_pi} and with \eqref{eq:cont_below_pi1}, we observe
$$
	\tau_1(z) < \tau_1(z') + (1-\varepsilon)\text{SB}(z,z')k < \tau_2(z),
$$	
	which contradicts the assumption $z$ is never occupied by type 1.
	
	Now consider the case $z\in s$ and $z\notin \pi$.
	By similar considerations to the previous case, 
\begin{equation}
\label{eq:cont_below_s}
	\tau_2(z) > \tau_1(z')+(1-\varepsilon)\left(\text{SB}(z,z')-2\right)^+k + k^3,
\end{equation}
where the $k^3$ term is because at least one edge incident to the spine $s$ must be traversed for type 2 to propagate to $z$.
	Note that the function SB only counts sub-boxes up to depth $kr$ in a box and not up to depth $kr+k^2$.
	Hence
\begin{equation}
\label{eq:cont_below_s1}
	T_1(z'\to z) < (1-\varepsilon)\text{SB}(z,z')k + (1-\varepsilon)k^2,
\end{equation}
where the $ (1-\varepsilon)k^2$ term is from the construction of the spine.
	By comparing \eqref{eq:cont_below_s} with \eqref{eq:cont_below_s1}, we have
$$
	\tau_1(z) < \tau_1(z') +(1-\varepsilon)\text{SB}(z,z')k + (1-\varepsilon)k^2 < \tau_2(z),
$$
which again gives a contradiction.
	
	Finally, consider the case $p_z\cap \partial V_{kr+k^2}^{\downarrow}(\sigma_{B_0})=\emptyset$ and $p(z)\notin B_0$ (see Figure~\ref{fig:progout2}).
	Under this assumption, $p_z$ contains $\sigma_{B_0}$.
	Moreover,  as $p(z)\notin B_0$, $p_z$ also contains a site neighbouring $\sigma_{B_0}$ in a spine of a good box that $\gamma$ passes through before entering $B_0$.
	Call this site $u$.
	As $u$ is contained in a spine of a good box, then
\begin{equation}
\label{eq:tau1_root}
\tau_1(\sigma_{B_0}) < \tau_1(u) + (1-\varepsilon)k^2,
\end{equation}
where we observe $\tau_1(u)<\infty$ as $\tau_1(\sigma_{B_0})<\infty$.
	The downwards passage time from $u$ to $\sigma_{B_0}$ is at least $k^3$ and thus
\begin{equation}
\label{eq:tau2_root}
\tau_2(\sigma_{B_0}) > \tau_2(u) + k^3 > \tau_1(u) + k^3 > \tau_1(\sigma_{B_0}) + k^3 - (1-\varepsilon)k^2,
\end{equation}
where in the final inequality we use \eqref{eq:tau1_root}.
	If $z\in\pi$,  by considering the type 2 passage times from $\sigma_{B_0}$ to $z$ along the highways, we deduce
\begin{equation}
\label{eq:cont_final}
	\tau_2(z) > \tau_1(\sigma_{B_0}) + k^3 - (1-\varepsilon)k^2 + (1-\varepsilon)\left(\text{SB}(\sigma_{B_0},z)-2\right)^+k.
\end{equation}
	By construction of the highway $\pi$, we have
\begin{equation}
\label{eq:cont_final1}
	T_1(\sigma_{B_0}\to z) < (1-\varepsilon)\text{SB}(\sigma_{B_0},z)k.
\end{equation}
	From \eqref{eq:cont_final} and \eqref{eq:cont_final1}, we have
$$
	\tau_1(z) < \tau_1(\sigma_{B_0}) + (1-\varepsilon)\text{SB}(\sigma_{B_0},z)k < \tau_2(z).
$$
	However, this contradicts the assumption $z$ is never occupied by type 1.
	A similar argument can be used if $z\in s$ and $z\notin \pi$ as in the previous case.
	As we have considered all possible cases for $z$ and $p(z)$, the contradiction is established and the proof is complete.
\end{proof}

\section{Behaviour on $\mathbb{Z}^d$}
\label{sec:lattice}

	In this section we consider the model on $\mathbb{Z}^d$.
	The model can be constructed in the exact same way as before except we no longer distinguish between upwards and downwards type 2 passage times.
	That is, we define $\left\{t_{1,e}\right\}_{e\in E(\mathbb{Z}^d)}$ and $\left\{t_{2,e}\right\}_{e\in E(\mathbb{Z}^d)}$ to be i.i.d.\ collections of exponentially distributed random variables of rate 1 and rate $\lambda$ on the edges of $\mathbb{Z}^d$, respectively.
	
	Type 1 spreads according to the $\left\{t_{1,e}\right\}_{e\in E(\mathbb{Z}^d)}$ passage times to vacant sites.
	Once a site $z$ is occupied by type 1, it attempts to convert to type 2 after waiting $\mathcal{I}_z$ time, where $\left\{\mathcal{I}_x\right\}_{x\in \mathbb{Z}^d}$ is an i.i.d.\ collection of exponentially distributed random variables of rate $\rho>0$.
	Type 2 spreads according to the $\left\{t_{2,e}\right\}_{e\in E(\mathbb{Z}^d)}$ passage times to vacant sites and sites occupied by type 1.	
	
\subsection{Proof of Theorem~\ref{thm:lattice}}
\label{sec:lattice1}

	In order to prove Theorem~\ref{thm:lattice}, we first recall some classical results for first passage percolation on $\mathbb{Z}^d$.
	
	Given sites $x,y\in\mathbb{Z}^d$, define the \emph{passage time from $x$ to $y$} as the random variable
$$
	T_1(x,y) = \inf_{\gamma\in\Gamma_{x,y}}T_1(\gamma)
$$
where $T_1$ is as defined in \eqref{eq:FPP_sum} and $\Gamma_{x,y}$ is the set of all finite paths from $x$ to $y$.
	For $t\geq0$, let
$$
	B(t) = \left\{x\in\mathbb{Z}^d:T_1(0,x)\leq t\right\}
$$
be the ball of radius $t$ centred at the origin under the metric induced by the exponential passage times of rate 1.
	The idea of the shape theorem is $B(t)$ converges to a deterministic shape once linearly rescaled in time.
	To make sense of rescaling, let 
$$
	\tilde{B}(t) = \left\{x+[-\tfrac{1}{2},\tfrac{1}{2})^d:T_1(0,x)\leq t\right\}
$$
be the set of sites in $B(t)$ considered as the centre of a unit cube in $\mathbb{R}^d$.

\begin{theorem}[Richardson \cite{rich}]
\label{thm:shape}
There exists a deterministic, convex, compact set $\mathcal{B}_1\subset\mathbb{R}^d$ such that for all $\varepsilon>0$
$$
	\mathbb{P}\left((1-\varepsilon)\mathcal{B}_1\subset\frac{\tilde{B}(t)}{t}\subset (1+\varepsilon)\mathcal{B}_1 \emph{ for all large }t\right)=1.
$$
The set $\mathcal{B}_1$ is called the limit shape.
\end{theorem}

	The proof of Theorem~\ref{thm:shape} relies on subadditivity arguments through Kingman's subadditive ergodic theorem \cite{kingman1973subadditive}.
	Consequently, the exact limiting shape is not known.
	Richardson's shape theorem has been extended to more general passage times by Cox and Durrett \cite{cox1981some} and recent results about the limit shape can be found in \cite{auffinger201750}.

	By time scaling, if one replaced rate 1 passage times with rate $\lambda$ passage times for some $\lambda>0$, then there would be a limit shape $\mathcal{B}_\lambda$ such that $\mathcal{B}_{\lambda}=\lambda\mathcal{B}_1$.
	
	To prove Theorem~\ref{thm:lattice}, we first need to prove the set of sites ever occupied by type 1 is contained in a first passage percolation process of rate strictly less than one.
	This is the content of the following lemma.
	Let $\eta_1(t)$ (resp.\ $\eta_2(t)$) be the set of sites ever occupied by type 1 (resp.\ type 2) up to time $t$.
	
\begin{lemma}
\label{lem:vdbK}
	Consider the model on $\mathbb{Z}^d$ with $d\geq2$ and $\rho,\lambda>0$.
	There exists $\kappa=\kappa(\rho,d)>0$ such that for all $\varepsilon>0$,
$$
	\eta_1(t) \subset (1+\varepsilon)(1-\kappa)t\mathcal{B}_1
$$
for all large enough $t$, almost surely.
\end{lemma}

\begin{proof}
	The proof is consequence of van den Berg and Kesten \cite{van1993inequalities} (see \cite[Proposition 6.4]{sidoravicius2019multi} for the result in terms of the limit shape).
	Their result states if a distribution $F$ strictly dominates\footnote{We say a distribution $F$ \emph{stricly dominates} a distribution 
	$\tilde F$ if there exists a coupling between these two distributions under which 
	the random variable with distribution $F$ is larger than the random variable with distribution $\tilde F$ with probability $1$.} 
	another distribution $\tilde{F}$, then the limiting shape under $F$ is strictly contained in the limiting shape under $\tilde{F}$, under some natural conditions on $F$ and $\tilde{F}$.
	The idea is to prove type 1 is contained in a process whose passage times strictly dominate a first passage percolation process with exponential passage times of rate 1.
	However, there are subtle dependencies arising because the spread of type 1 is facilitated by the interplay of small type 1 passage times and large conversion times.
	We need to prove we can construct the model in a manner that decouples these two sources of randomness in order to apply a van den Berg--Kesten argument.

	For $K>0$, let $u(K)$ be the probability an exponential random variable of rate $\rho$ is at least $K$.
	Given $x\in\mathbb{Z}^d$, let $x$ be \emph{marked} with probability $1-u(K)$, and \emph{unmarked} with probability $u(K)$, independently of every other site.
	Define an edge $e$ to be marked if both of its endpoints are marked.
	The set of marked edges gives a 1-dependent percolation process.
	By Liggett, Schonmann and Stacey \cite{liggett1997domination}, for any $q\in(0,1)$, the set of marked edges stochastically dominates an i.i.d.\ Bernoulli percolation process of parameter $q$ for a large enough choice of $K$.
	The open edges according to this i.i.d.\ Bernoulli percolation process are defined to be \emph{semi-marked}.
	
	For each edge $e$, determine whether $e$ is semi-marked and sample a candidate type 1 passage time $\tilde{t}_{1,e}$ that is an exponential random variable of rate 1, independently of every other edge.
	If $e$ is not semi-marked, we let $f_{1,e}=\tilde{t}_{1,e}$.
	If $e$ is semi-marked, we set
$$
	f_{1,e} = \begin{cases}
	\tilde{t}_{1,e} & \mbox{if }\tilde{t}_{1,e}\leq K,\\
	\infty & \mbox{if }\tilde{t}_{1,e}> K.
	\end{cases}
$$
	Hence the evolution for type 1 is identical under $\left\{f_{1,e}\right\}_e$ and $\left\{\tilde{t}_{1,e}\right\}_e$ by this construction.
	The passage times given by $\left\{f_{1,e}\right\}_e$ are i.i.d.\ and are stochastically dominated by i.i.d.\ exponential passage times of rate 1.
	Hence we may apply van den Berg--Kesten and the result follows from Theorem~\ref{thm:shape}.
\end{proof}

	With Lemma~\ref{lem:vdbK} to hand, we are now in a position to prove Theorem~\ref{thm:lattice}.
	
\begin{proof}[Proof of Theorem~\ref{thm:lattice}]
	Let $\kappa$ be as in Lemma~\ref{lem:vdbK} so for any $\varepsilon>0$, 
$$
	\eta_1(t) \subset (1+\varepsilon)(1-\kappa)t\mathcal{B}_1,
$$
for all $t$ large enough, almost surely.
	By Theorem~\ref{thm:shape}, if we only consider the evolution of type 2 after the origin has been converted, for any $\varepsilon>0$, 
$$
	\eta_2(t)\supset (1-\varepsilon)t\mathcal{B}_\lambda = (1-\varepsilon)\lambda t \mathcal{B}_1,
$$
for all $t$ large enough, almost surely.
	If $\lambda>1-\kappa$ and we fix $\varepsilon$ small enough so
$$
	(1+\varepsilon)(1-\kappa) < (1-\varepsilon)\lambda,
$$
we deduce type 1 dies out almost surely as $\eta_1(t)\subset\eta_2(t)$ for all large enough $t$.
\end{proof}

\subsection{Sidoravicius--Stauffer percolation}
\label{sec:SSP}

	The proof of Theorem~\ref{thm:rho} relies on a coupling between our converting first passage percolation model with a random competition process called \emph{Sidoravicius--Stauffer percolation} (SSP).
	SSPs where introduced by Dauvergne and Sly \cite{dauvergne2021spread} as a streamlined version of FPPHE from \cite{sidoravicius2019multi}.
	The purpose of this section is to define SSPs and outline an important encapsulation theorem by Dauvergne and Sly.
	
	An SSP on $\mathbb{Z}^d$ consists of two competing growth processes $\mathfrak{R}(t)$ and $\mathfrak{B}(t)$ for $t\geq0$, called red and blue for clarity, respectively.
	Let $\overrightarrow{\text{E}}$ be the set of directed edges on $\mathbb{Z}^d$, so that
$$
	\overrightarrow{\text{E}} = \left\{(u,v):(u,v)\in E\left(\mathbb{Z}^d\right)\right\}.
$$
	To define an SSP we require the following.
\begin{itemize}
\item Functions $X_{\mathfrak{R}}:\overrightarrow{\text{E}}\to [0,1]$ and $X_{\mathfrak{B}}:\overrightarrow{\text{E}}\to [0,\infty)$ viewed as the clocks defining the growth of the red and blue process, respectively.
\item A collection of blue seeds $\mathfrak{B}_*\subset \mathbb{Z}^d$.
\item A parameter $\kappa>1$ such that $X_{\mathfrak{B}}(u,v)\leq \kappa$ for all $(u,v)\in \overrightarrow{\text{E}}$.
\end{itemize} 
	If $u$ and $v$ are both blue seeds, we take $X_{\mathfrak{B}}(u,v)=0$.
	That is, the blue process spreads instantaneously through connected components of blue seeds.
	For our purposes, the blue process spreads after waiting time $\kappa$ through every other edge.
	Hence we will always consider the special case where the blue process spreads according to the clocks given by
$$
	X_{\mathfrak{B}}:\overrightarrow{\text{E}}\to [0,\kappa] \mbox{ with } X_{\mathfrak{B}}(u,v) = \begin{cases} 0 & \mbox{if } u,v\mbox{ are blue seeds},\\
	\kappa & \mbox{otherwise}.
	\end{cases}
$$

	At time $t=0$, the red process only occupies the origin while the blue process is dormant in seeds.
	SSPs evolve in time through the following dynamics.
	Given a site $u$, let $T(u)$ be the earliest time the red or blue process occupied $u$ and $C(u)\in\{\mathfrak{R},\mathfrak{B}\}$ denote the colour of $u$ once it is occupied.
	Given an edge $(u,v)$, at time $T(u)+X_{C(u)}(u,v)$, the edge $(u,v)$ will ring and the process evolves in the following manner.
	If $T(v)<T(u)+X_{C(u)}(u,v)$, then the occupation is suppressed as $y$ has already been coloured by an invasion from another edge.
	Otherwise, we colour $v$ according to the following rules:
\begin{itemize}
\item If $C(u)=\mathfrak{R}$ and $v\in\mathfrak{B}_*$, then $C(v)=\mathfrak{B}$.
\item If $C(u)=\mathfrak{R}$ and $v\notin\mathfrak{B}_*$, then $C(v)=\mathfrak{R}$.
\item If $C(u)=\mathfrak{B}$, then $C(v)=\mathfrak{B}$.
\end{itemize}

	The construction of SSPs in \cite{dauvergne2021spread} allows for the red process to invade through blue sites in some circumstances due to caveats in their application.
	They couple SSPs with a spread of infection model and blue regions are only where they cannot guarantee the infection is moving fast enough.
	For our application, we do not need to consider these cases and so the above construction will suffice.
	
	We define the red (resp.\ blue) process to \emph{survive} if there is an infinite connected region of red (resp.\ blue) sites in the limit as $t\to\infty$.
	Otherwise we say the red (resp.\ blue) process \emph{dies out}.
	
	The following is an encapsulation result of Dauvergne and Sly \cite[Theorem 2.14]{dauvergne2021spread}, where the red process survives and encapsulates all blue regions with positive probability, 
	so long as $\kappa$ is large enough (so the blue process is substantially slower than the red process) and blue seeds are stochastically dominated 
	by an i.i.d.\ Bernoulli process of small enough parameter.

\begin{theorem}
\label{thm:enc_SS}
	Consider a random SSP $(\mathfrak{R},\mathfrak{B})$ on $\mathbb{Z}^d$ driven by potentially random clocks $X_{\mathfrak{R}},X_{\mathfrak{B}}$,  a collection of blue seeds $\mathfrak{B}_*$ and a constant parameter $\kappa>4000$.
	Suppose additionally $\mathfrak{B}_*$ is stochastically dominated by an i.i.d.\ Bernoulli process of parameter $p>0$.
	There exists a universal constant $c>0$ such that the probability the red process survives and the blue process dies out is at least $1-cp$.
\end{theorem}

	The statement of Theorem~\ref{thm:enc_SS} suffices for our purposes and is provided in a more detailed manner in \cite{dauvergne2021spread}.

\subsection{Proof of Theorem~\ref{thm:rho}}

	In this section we prove Theorem~\ref{thm:rho} and in doing so, partially answer Conjecture~\ref{conj:lattice}.
	We first prove $\rho_{\ell}>0$ through a coupling with an appropriate SSP.
	To facilitate this coupling, we need to provide an alternative construction of the model on $\mathbb{Z}^d$.
	
	Type 1 evolves according the passage times $\left\{t_{1,e}\right\}_{e\in E(\mathbb{Z}^d)}$ and converts to type 2 according to the conversion times $\left\{\mathcal{I}_x\right\}_{x\in\mathbb{Z}^d}$, as before.
	When a site $x$ is occupied by type 2, it attempts to spread to neighbouring sites, according to the passage times $\left\{t_{2,e}\right\}_{e\in E(\mathbb{Z}^d)}$.
	That is, if $x$ is first occupied by type 1 at time $s$ and $y\sim x$ is vacant at time $s+t_{2,xy}$, then $y$ is occupied by type 2 at time $s+t_{2,xy}$.
	If $y$ is occupied by type 1 at some time $s'\in[s,s+t_{2,xy})$, type 2 now attempts to occupy $x$ at time $s'+t_{3,xy}$ where $\left\{t_{3,e}\right\}_{e\in E(\mathbb{Z}^d)}$ is an i.i.d.\ collection of exponential random variables of rate $\lambda$.
	By the memoryless property of the exponential distribution, this construction is equivalent to how the model is defined before.
	
\begin{proof}[Proof of Theorem~\ref{thm:rho}: $\rho_{\ell}>0$.]
	Let $C$ be a large constant we fix later.
	Define a site $x$ to be a \emph{type 2-seed} if at least one of the following holds.
\begin{itemize}
\item There exists $y\sim x$ such that $t_{1,xy}\geq C$.
\item There exists $y\sim x$ such that $t_{2,xy}<C^2$.
\item There exists $y\sim x$ such that $t_{3,xy}<C^2$.
\item The conversion time at $x$ satisfies $\mathcal{I}_x<C^2$.
\end{itemize}

	We now construct the appropriate SSP required for the coupling argument.
	The blue seeds for the SSP are given by the type 2-seeds above.
	Define the clocks $X_{\mathfrak{R}}$ and $X_{\mathfrak{B}}$ governing the spread of the red and blue process as follows:
$$
	X_{\mathfrak{R}}:\overrightarrow{\text{E}}\to [0,C] \mbox{ with } X_{\mathfrak{R}}(x,y) = \min\{t_{1,xy},C\},
$$
and 
$$
	X_{\mathfrak{B}}:\overrightarrow{\text{E}}\to [0,C^2] \mbox{ with } X_{\mathfrak{B}}(x,y) = \begin{cases} 0 & \mbox{if } x,y\mbox{ are blue seeds},\\
	C^2 & \mbox{otherwise}.
	\end{cases}
$$
	Note the change from $X_\mathfrak{R}$ being bounded by 1 to being bounded by $C$ amounts to a time change and does not alter any arguments.
	The only edges where the red clock is equal to $C$ must have a type 2-seed at each endpoint and thus play no role in the evolution of the red process.
	
	We deduce if a site $x$ is not a type 2-seed and occupied by type 1, it is able to attempt to spread type 1 to all of its neighbours. 
	For example, if $y\sim x$, then
$$
	t_{1,xy}<\min_{z\sim x}\left\{t_{2,xz},t_{3,xz},\mathcal{I}_x\right\}
$$
	 via the construction of type 2-seeds and so the propagation of type 1 from $x$ cannot be blocked by the spread of type 2 or conversions.
	 This is precisely why we needed to define the passage times $\left\{t_{3,e}\right\}_{e\in E(\mathbb{Z}^d)}$, so the spread of type 2 from a neighbouring site does not block the spread of type 1.
	
	If there where no type 2-seeds, it is immediate type 1 survives as type 2 does not have the potential to block its spread.
	Through this construction, type 2 only can block type 1 through the spread from type 2-seeds.
	Sites that can potentially be blocked from type 1 through type 2-seeds are then the blue process while the remaining sites are red.
	Hence, proving the red process survives implies type 1 survives.

	It is easy to verify this construction yields a valid SSP in the language from \cite{dauvergne2021spread}.
	Moreover, by setting $C$ large enough and then $\lambda$ and $\rho$ small enough, the probability a site is a type 2-seed can be made arbitrarily small.
	The process of labelling sites as type 2-seeds defines a 1-dependent percolation process and so can be constructed to be stochastically dominated by an i.i.d.\ Bernoulli process of parameter $p$, for any $p\in(0,1)$ by Liggett, Schonmann and Stacey \cite{liggett1997domination}.
	This observation allows us to deduce from Theorem~\ref{thm:enc_SS} that for $\lambda$ and $\rho$ small enough, the red process survives and occupies infinitely many sites, and all connected components of the blue process are finite, with positive probability.
\end{proof}

	The proof that $\rho_u$ exists and is finite requires less machinery.

\begin{proof}[Proof of Theorem~\ref{thm:rho}: $\rho_u<\infty$]
%
	Define a site $x$ to be \emph{closed} if the minimum type 1 passage time on an edge incident to $x$ is greater than the time it takes $x$ to convert once occupied by type 1, so that
$$
	\min_{y\sim x}t_{1, xy} > \mathcal{I}_x.
$$ 
   Note that closed sites cannot pass type 1 to any of their neighbors.
	The process of labelling sites closed is a 1-dependent percolation process with
$$
	\lim_{\rho\to \infty}\mathbb{P}\left(x\mbox{ is closed}\right)=1.
$$
Through Liggett, Schonmann and Stacey \cite{liggett1997domination}, by setting $\rho$ large enough, we deduce closed sites stochastically dominate a supercritical i.i.d.\ Bernoulli percolation process. 
	Hence, for all large enough $\rho$, type 1 dies out almost surely as the origin is encapsulated by closed sites.
\end{proof}

\appendix

\section{Appendix: Standard large deviation results}

\begin{lemma}[Chernoff bounds for Poisson random variables]
\label{lem:chernoff}
	Let $P$ be a Poisson random variable of mean $\mu$.
	For any $\varepsilon\in(0,1)$,
$$
	\mathbb{P}\left(P < (1-\varepsilon)\mu\right) < \exp\left\{-\mu \varepsilon^2/2\right\}
$$
and 
$$
	\mathbb{P}\left(P > (1+\varepsilon)\mu\right) < \exp\left\{-\mu \varepsilon^2/4\right\}.
$$
For any $C>0$ and $\theta\in\mathbb{R}$, we have
$$
	\mathbb{P}\left(P > C\mu\right) \leq \exp\left\{-\mu\left(1-e^{\theta}+\theta C\right)\right\}.
$$
\end{lemma}

\bibliographystyle{abbrv}

\bibliography{references}

\end{document}